\documentclass[12pt]{amsart}
\usepackage{epsfig,color}

\theoremstyle{definition}

\newtheorem{theorem}{Theorem}[section]
\newtheorem{definition}[theorem]{Definition}
\newtheorem{conjecture}[theorem]{Conjecture}

\newtheorem{lemma}[theorem]{Lemma}
\newtheorem{remark}[theorem]{Remark}
\newtheorem{corollary}[theorem]{Corollary}

\numberwithin{equation}{section}

\newcommand{\abs}[1]{\lvert#1\rvert}

\begin{document}

\title{Self-Shrinkers With Second Fundamental Form of Constant Length}

\author{Qiang Guang}

\address{}

\curraddr{}
\email{qguang@math.mit.edu}   
\thanks{}

\tolerance=1
\emergencystretch=\maxdimen
\hyphenpenalty=10000
\hbadness=10000

\begin{abstract}
In this note, we give a new and simple proof of a result in {\cite{DX1}} which states that any smooth complete self-shrinker in $\mathbb{R}^3$ with second fundamental form of constant length must be a generalized cylinder $\mathbb{S}^k \times \mathbb{R}^{2-k}$ for some $k\leq2$. Moreover, we prove a gap theorem for smooth self-shrinkers in all dimensions.  
\end{abstract}

\maketitle

\section{Introduction}
A one-parameter family of hypersurfaces $M_t \subset \mathbb{R}^{n+1}$ flows by mean curvature if
\begin{equation}
\partial_{t}x=-H\mathbf{n},
\end{equation}
where $H$ is the mean curvature, $\mathbf{n}$ is the outward pointing unit normal and $x$ is the position vector.

We call a hypersurface ${\Sigma}^{n} \subset \mathbb{R}^{n+1}$ a $self$-$shrinker$ if it satisfies
\begin{equation}
H=\frac{\langle x,\mathbf{n}\rangle}{2}.
\end{equation}
Under the mean curvature flow (``MCF"), $\Sigma$ is just moving by rescalings, i.e., $\Sigma_t\equiv \sqrt{-t}\Sigma$ gives a MCF. 

Self-shrinkers play a key role in the study of MCF. By Huisken's montonicity formula \cite{Hui90} and an argument of Ilmanen \cite{I1} and White, self-shrinkers provide all singularity models of the MCF. Although there are infinitely many of them, we only know few embedded complete examples, see \cite{A}, \cite{Ch}, \cite{KM}, \cite{M1} and \cite{N1}. Moreover, numerical results showed that it is impossible to give a complete classification of self-shrinkers in higher dimensions. However, under certain conditions, there are many classification results of self-shrinkers. In {\cite{CM1}}, Colding and Minicozzi proved that the only smooth complete embedded self-shrinkers with polynomial volume growth and $H\geq 0$ in $\mathbb{R}^{n+1}$ are generalized cylinders $\mathbb{S}^k \times \mathbb{R}^{n-k}$ (where the $\mathbb{S}^k$ has radius $\sqrt{2k}$). A natural question is under what other conditions can we conclude that a self-shrinker is a generalized cylinder. We are interested in those conditions involving the squared norm of the second fundamental form, i.e., $\abs{A}^2$. First, we consider the following question.

\begin{conjecture}\label{3}
${\Sigma}^{n} \subset \mathbb{R}^{n+1}$ is a smooth complete embedded self-shrinker with polynomial volume growth. If $\abs{A}^2=constant$, then $\Sigma$ is a generalized cylinder.
\end{conjecture}
The case $n=1$ follows from a more general result by  Abresch and Langer \cite{AbL} stating that the only smooth complete and embedded self-shrinkers in $\mathbb{R}^2$ are the lines and a round circle. In $\mathbb{R}^3$, i.e., $n=2$, the above conjecture was proved by Ding and Xin \cite{DX1} using the following identity

\begin{equation}\label{5}
\frac{1}{2} \mathcal{L}\abs{\nabla A }^2=\abs{ {\nabla}^2 A }^2+(1-\abs{A}^2)\abs{\nabla A }^2-3 \Xi-\frac{3}{2}\abs{\nabla \abs{A}^2 }^2
\end{equation}
where the operator $\mathcal{L}=\Delta-\langle\frac{x}{2},\nabla \cdot\rangle$, $h_{ij}$ is the second fundamental form and $\Xi=\sum_{i,j,k,l,m} h_{ijk}h_{ijl}h_{km}h_{ml}-2 \sum_{i,j,k,l,m}h_{ijk}h_{klm}h_{im}h_{jl}$.

In this note, we give a new and simple proof of the above result without heavy computation, more precisely, we prove the following theorem.

\begin{theorem}\label{8}
Let $\Sigma^2 \subset \mathbb{R}^3$ be a smooth complete embedded self-shrinker with polynomial volume growth. If the second fundamental form of $\Sigma^2$ is of constant length, i.e., ${\lvert A \rvert}^2$=constant, then $\Sigma^2$ is a generalized cylinder $\mathbb{S}^k \times \mathbb{R}^{2-k}$ for $k \leq 2$.
\end{theorem}

The key idea in the proof is to analyze the point where  $\lvert x \rvert$ achieves its minimum. We mention that our method does not apply to higher dimensions to prove the Conjecture \ref{3}. \\

For self-shrinkers, there exists some gap phenomenon for the squared norm of the second fundamental form. Cao and Li \cite{CL1} proved that any smooth complete self-shrinker with polynomial volume growth and $\abs{A}^2 \leq \frac{1}{2}$ in arbitrary codimension is a generalized cylinder. Colding, Ilmanen and Minicozzi \cite{CM5} showed that generalized cylinders are rigid in a strong sense that any self-shrinker which is sufficiently close to one of generalized cylinders on a large and compact set must itself be a generalized cylinder. Using this result we prove that any self-shrinker with $\abs{A}^2$ sufficiently close to $\frac{1}{2}$ must also be a generalized cylinder.

\begin{theorem}\label{10}
Given $n$ and $\lambda_0$, there exists $\delta=\delta(n,\lambda_0)>0$ so that if $ \Sigma^n \subset \mathbb{R}^{n+1}$ is a smooth embedded self-shrinker with entropy $\lambda(\Sigma) \leq \lambda_0$ and
\begin{itemize}
\item $\abs{A}^2 \leq \frac{1}{2}+\delta$,
\end{itemize}
then $\Sigma^n$ is a generalized cylinder $\mathbb{S}^k \times \mathbb{R}^{n-k}$ for some $k\leq n$.
\end{theorem}
\begin{remark}
It is expected that we could remove the entropy bound in Theorem \ref{10} in dimensional two case. In other words,  the bound for $\abs{A}^2$ may imply the entropy bound in  $\mathbb{R}^3$. Note that in $\mathbb{R}^3$, if $\Sigma^2$ is a closed self-shrinker with ${\lvert A \rvert}^2 \leq C$, where $C$ is a constant less than 1. Then by \emph{Gauss-Bonnet Formula}, one can easily get that the genus of $\Sigma^2$ is 0 and therefore obtain an entropy bound for $\Sigma^2$. 
\end{remark}

\section{Background and Preliminaries}
In this section, we recall some background and preliminaries for self-shrinkers from \cite{CM1}.  Throughout this note, we always assume self-shrinkers to be smooth complete embedded, without boundary and with polynomial volume growth. 

Let $\Sigma \subset \mathbb{R}^{n+1}$ be a hypersurface, then $\Delta$, div, and $\nabla$ are the Laplacian, divergence, and gradient, respectively, on $\Sigma$. $\mathbf{n}$ is the outward unit normal, $H$=div$_{\Sigma}\mathbf{n}$ is the mean curvature, $A$ is the second fundamental form, and $x$ is the position vector. With this convection, the mean curvature $H$ is $n/r$ on the sphere $\mathbb{S}^n\subset \mathbb{R}^{n+1}$ of radius $r$. 

First, recall the operators $\mathcal{L}$ and $L$ defined by 
\begin{equation}\label{12}
\mathcal{L}=\Delta-\frac{1}{2}\langle x, \nabla \cdot \rangle,
\end{equation}
\begin{equation}\label{13}
L=\Delta-\frac{1}{2}\langle x, \nabla \cdot \rangle+\abs{A}^2+\frac{1}{2}.
\end{equation}

\begin{lemma}[{\cite{CM1}}]\label{15}
If $\Sigma^n \subset \mathbb{R}^{n+1}$ is a smooth self-shrinker, then 
\begin{equation}\label{16}
\mathcal{L}\abs{x}^2=2n-\abs{x}^2,
\end{equation}
\begin{equation}\label{17}
\mathcal{L}H^2=2(\frac{1}{2}-\abs{A}^2)H^2+2\abs{\nabla H}^2,
\end{equation}
\begin{equation}\label{18}
\mathcal{L}\abs{A}^2=2(\frac{1}{2}-\abs{A}^2)\abs{A}^2+2\abs{\nabla A}^2.
\end{equation}
\end{lemma}

A direct consequence of Lemma \ref{15} is the following corollary.

\begin{corollary}\label{20}
Let $\Sigma^n \subset \mathbb{R}^{n+1}$ be a smooth self-shrinker. If $\abs{A}^2\leq\frac{1}{2}$, then $\Sigma$ is a generalized cylinder $\mathbb{S}^k \times \mathbb{R}^{n-k}$ for some $k\leq n$. Moreover, if $\abs{A}^2<\frac{1}{2}$, then $\Sigma$ is a hyperplane.

\end{corollary}

Next, we introduce the concepts of the $F$-functional and the entropy of a hypersurface.
\begin{definition}\label{21}
For $t_0 >0$ and $x_0 \in \mathbb{R}^{n+1}$, 
the $F$-functional $F_{x_0,t_0}$ of a hypersurface $M\subset\mathbb{R}^{n+1}$ is defined as
\begin{equation}\label{22}
F_{x_0,t_0}(M)=(4\pi t_0)^{-\frac{n}{2}} \int_{M} {\textrm{e}^{-\frac{|x-x_0|^2}{4t_0}}},
\end{equation}
and the entropy of $M$ is given by 
\begin{equation}\label{23}
\lambda(M)=\sup_{x_0,t_0} F_{x_0,t_0}(M),
\end{equation}
here the supremum is taking over all $t_0 >0$ and $x_0 \in \mathbb{R}^{n+1}$.
\end{definition}

\section{Proof of Theorem \ref{8}}
\begin{proof}
By Corollary \ref{20}, if $\abs{A}^2<\frac{1}{2}$, then $\Sigma$ is a hyperplane in $\mathbb{R}^3$. Therefore, in the following we only consider the case when $\abs{A}^2\geq\frac{1}{2}$.

For any point $p\in\Sigma$, we can choose a local orthonormal frame $\{e_1,e_2\}$ such that the coefficients of the second fundamental form $h_{ij}=\lambda_i \delta_{ij}$ for $i,j=1,2$, then 
\begin{equation}\label{25}
\abs{\nabla A}^2=h_{111}^2+h_{222}^2+3h_{112}^2+3h_{221}^2=\abs{A}^2(\abs{A}^2-\frac{1}{2}),
\end{equation}
\begin{equation}\label{26}
\abs{\nabla H}^2=(h_{111}+h_{221})^2+(h_{112}+h_{222})^2.
\end{equation}
Since  $\abs{A}^2=$constant, we have
\begin{equation}\label{27}
h_{11}h_{111}+h_{22}h_{221}=h_{11}h_{112}+h_{22}h_{222}=0.
\end{equation}

First, we prove $\abs{x}>0$ on $\Sigma$. We argue by contraction. Suppose $\Sigma$ goes through the origin, then at the origin, we have $H=\abs{\nabla H}=0$, therefore
\begin{equation*}
h_{11}+h_{22}=h_{111}+h_{221}=h_{112}+h_{222}=0.
\end{equation*}
Combining this with (\ref{27}), we get
\begin{equation*}
h_{111}=h_{222}=h_{112}=h_{221}=0.
\end{equation*}
This implies that $\abs{\nabla A}^2=\abs{A}^2(\abs{A}^2-\frac{1}{2})=0$, i.e., $\abs{A}^2=\frac{1}{2}$. By Corollary \ref{20}, we conclude that $\Sigma$ is $\mathbb{S}^2$ or $\mathbb{S}^1 \times \mathbb{R}$. However, this contradicts the assumption that $\Sigma$ goes through the origin.\\

Note that $\Sigma$ has polynomial volume growth implies that $\Sigma$ is proper (see Theorem 4.1 in \cite{CZ}) and by the maximum principle $\Sigma$ intersects $\mathbb{S}^2(2)$, so there exists a point $p\in \Sigma$ minimize $\lvert x \rvert$. 

Now, at point $p$, we have $\abs{x}>0 $ and  $x^T=0$, where $x^T$ is the tangential projection of $x$. This implies that $4{H}^2(p)=\abs{x}^2(p)$, $\nabla H(p)=0$, and thus
\begin{equation*}
h_{111}+h_{221}=h_{112}+h_{222}=0,
\end{equation*}
By (\ref{27}), we get
\begin{equation}
h_{111}(h_{11}-h_{22})=h_{222}(h_{11}-h_{22})=0.
\end{equation}

If $h_{111}=h_{222}=0$, then $\abs{\nabla A}^2=0$ and therefore $\abs{A}^2=\frac{1}{2}$. By Corollary \ref{20}, we conclude that $\Sigma$ is a generalized cylinder.

If $h_{11}=h_{22}$, then we have 
\begin{equation}\label{30}
\abs{A}^2=2h_{11}^2=\frac{H^2(p)}{2}=\frac{\abs{x}^2(p)}{8}.
\end{equation}
Since every smooth complete self-shrinker must intersect the sphere $\mathbb{S}^2(2)$, we conclude that $\abs{x}(p)\leq 2$. By (\ref{30}), this gives
\begin{equation*}
\abs{A}^2\leq \frac{1}{2},
\end{equation*} 
then the theorem follows immediately from Corollary \ref{20}.

\end{proof}

\section{Proof of Theorem \ref{10}}
First, we state two key ingredients from \cite{CM5}. The first one is the rigidity theorem for the generalized cylinders and the second one is the compactness theorem for self-shrinkers.

\begin{theorem}[{\cite{CM5}}]\label{33}
Given $n,\lambda_0$ and $C$, there exists $R=R(n,\lambda_0,C)$ so that if $\Sigma^n \subset \mathbb{R}^{n+1}$ is a self-shrinker with entropy $\lambda(\Sigma) \leq \lambda_0$ satisfying 
\begin{itemize}
\item $\Sigma$ is smooth in $B_R$ with $ H \geq 0$ and $\abs{A} \leq C$ on $B_R \cap \Sigma$,
\end{itemize}
then $\Sigma$ is a generalized cylinder $\mathbb{S}^k \times \mathbb{R}^{n-k}$ for some $k \leq n$.
\end{theorem}

\begin{lemma}[{\cite{CM5}}]\label{34}
Let $\Sigma_i \subset \mathbb{R}^{n+1}$ be a sequence of $F$-stationary varifolds with $\lambda(\Sigma_i) \leq \lambda_0$ and
\begin{equation}
B_{R_i} \cap \Sigma_i \textrm{ is smooth with } \abs{A} \leq C,
\end{equation}
where $R_i \rightarrow \infty$. Then there exists a subsequence $\Sigma_i'$ that converges smoothly and with multiplicity one to a complete embedded self-shrinker $\Sigma$ with $\abs{A} \leq C$ and
\begin{equation}
\lim_{i \rightarrow \infty} \lambda(\Sigma_i')=\lambda(\Sigma).
\end{equation}
\end{lemma}

\begin{remark}
In the above lemma, the entropy bound is used to guarantee that the convergence is with finite multiplicity. Moreover, if the multiplicity is great than one ,then the limit is $L$-stable. Using the fact that there are no complete $L$-stable self-shrinkers with polynomial volume growth  gives the multiplicity of the convergence must be one. 
\end{remark}

Now we are ready to give the proof of Theorem \ref{10}.

\begin{proof}[Proof of Theorem \ref{10}]
We will argue by contradiction, so suppose there is a sequence of smooth embedded self-shrinkers $\Sigma_i \neq \mathbb{S}^k \times \mathbb{R}^{n-k}$ ($k \leq n$) with $\lambda(\Sigma_i) \leq \lambda_0$ and
\begin{equation}\label{35}
\abs{A}^2 \leq \frac{1}{2}+\frac{1}{i}.
\end{equation}

By Lemma \ref{34} there exists a subsequence $\Sigma_i$ (still denoted by $\Sigma_i$) that converges smoothly and with multiplicity one to a complete embedded self-shrinker $\Sigma$.\\
By (\ref{35}), we can conclude that $\Sigma$ satisfies $\abs{A}^2 \leq \frac{1}{2}$, and thus, $\Sigma$ is a generalized cylinder.\\
Now we choose the $R$ in Theorem \ref{33}, and for $N$ sufficiently large, $\Sigma_m$ is very close to $\Sigma$ on $B_{2R} \cap \Sigma_m$ for $m \geq N$, i.e.,
\begin{equation*}
\Sigma_m \textrm{ satisfies } H \geq 0 \textrm{ and } \abs{A} \leq 1 \textrm{ on } B_R \cap \Sigma_m,
\end{equation*}
then by the rigidity theorem of self-shrinkers, Theorem \ref{33}, $\Sigma_m$ is a generalized cylinder. However, this contradicts our assumption that $\Sigma_i$ is not a generalized cylinder, completing the proof.
\end{proof}

\bibliographystyle{alpha}
\bibliography{REF}

\end{document}